\documentclass[pdflatex,sn-mathphys-num]{sn-jnl}

\usepackage{graphicx}%
\usepackage{multirow}%
\usepackage{amsmath,amssymb,amsfonts}%
\usepackage{amsthm}%
\usepackage{mathrsfs}%
\usepackage[title]{appendix}%
\usepackage{xcolor}%
\usepackage{textcomp}%
\usepackage{manyfoot}%
\usepackage{booktabs}%
\usepackage{algorithm}%
\usepackage{algorithmicx}%
\usepackage{algpseudocode}%
\usepackage{listings}%
\usepackage{subcaption}

\theoremstyle{thmstyleone}%
\newtheorem{theorem}{Theorem}
\theoremstyle{thmstyletwo}%
\newtheorem{example}{Example}%
\newtheorem{remark}{Remark}%

\theoremstyle{thmstylethree}%

\begin{document}

\title[LMEP Unification]{Unifying Finite Differences and Semi-Lagrangian Schemes via Localized Matrix Exponentials}

\author{\fnm{V\'ictor} \sur{Bayona}}\email{vbayona@math.uc3m.es}

\affil{\orgname{Universidad Carlos III de Madrid, ROR: https://ror.org/03ths8210}, \orgdiv{Departamento de Matem\'aticas}, \orgaddress{\street{Avenida de la Universidad, 30}, \city{Legan\'es}, \postcode{28911}, \state{Madrid}, \country{Espa\~na}}}

\abstract{We present a unified framework for the construction of localized exponential integrators that bypasses the traditional trade-off between the accuracy of global spectral methods and the efficiency of sparse finite differences. By evaluating the matrix exponential of a discrete operator strictly within a local stencil of size $n$, we "harvest" integration weights that naturally incorporate high-order temporal corrections. We prove that this Local Matrix Exponential Propagator (LMEP) is algebraically isomorphic to optimal semi-Lagrangian transport for advection and provides algebraically exact coupled evolution for mixed-physics operators, effectively eliminating the commutator errors associated with operator splitting. The framework is extended to semi-linear systems via a localized augmented matrix approach, facilitating the evaluation of Exponential Time Differencing (ETD) $\phi$-functions through sparse, banded operations. Numerical experiments on the viscous Burgers, Korteweg-de Vries, and Allen-Cahn equations demonstrate that the method preserves high-order temporal accuracy and exhibits  superior stability at high Courant numbers across both periodic and non-periodic domains. We empirically demonstrate that this localized approach yields optimal $\mathcal{O}(N)$ scaling and, for high-CFL upwind configurations, total execution times that remain strictly independent of the spatial approximation order.}


\maketitle

\section{Introduction}
\label{sec:intro}

The numerical simulation of multi-physics phenomena, ranging from geophysical fluid dynamics to nonlinear materials science, necessitates the integration of partial differential equations (PDEs) that are often characterized by severe temporal stiffness and complex spatial coupling. The fundamental challenge in designing such integrators lies in balancing three often-conflicting requirements: the high-order spatial accuracy of global spectral methods, the strict $O(N)$ computational efficiency of local stencils, and the robust stability required to evolve stiff operators with competitive time steps.

Traditionally, practitioners have addressed these challenges through two primary paradigms. Global spectral methods provide exponential convergence for periodic and simple geometries \cite{Trefethen2000}, yet they result in dense differentiation matrices that scale poorly ($O(N^2)$ or $O(N \log N)$) and lack the geometric flexibility required for modern engineering applications. Conversely, operator-splitting techniques, such as the ubiquitous Strang splitting \cite{Strang1968}, decompose complex evolution operators into simpler components that can be solved efficiently. However, as analyzed in recent research on nonlinear evolution equations \cite{Siqi2025}, these methods introduce persistent "commutator" errors. In high-accuracy regimes, these splitting errors frequently dominate the total numerical uncertainty, particularly when coupling advection-dominated transport with stiff diffusion or dispersion \cite{Hundsdorfer2003}.

Exponential Time Differencing (ETD) methods \cite{Hochbruck2010} have emerged as a powerful third paradigm. By integrating the linear part of a PDE exactly via the matrix exponential, ETD schemes exhibit superior stability and preserve physical invariants more effectively than standard IMEX or Runge-Kutta methods \cite{Cox2002}. Despite their theoretical advantages, global ETD is hindered by the "fill-in" phenomenon: the exponential of a sparse differentiation matrix is typically dense. While recent advances have improved the parallel efficiency of matrix-vector actions for large-scale systems \cite{Wu2021,Li2022_Calcolo,Blanes2026}, the computational cost remains a significant barrier for non-periodic problems and unstructured grids.

To mitigate these constraints, significant progress has been made in developing localized ETD methods, particularly those leveraging domain decomposition (DD) frameworks. Recent works have introduced both overlapping \cite{Hoang2018,Li2021} and nonoverlapping \cite{Hoang2020} DD strategies, which facilitate parallel integration for diffusion and semi-linear parabolic equations. These approaches effectively distribute the computational burden but often rely on iterative interface solvers or Schwarz waveform relaxation to maintain accuracy at subdomain boundaries. Parallel efforts in radial basis functions (RBF-FD) have also sought to "harvest" integration weights to maintain mass conservation in localized settings \cite{Fornberg2015, Qiao2025}. However, many localized methods treat the discretization as a numerical approximation of a global operator, rather than a unified algebraic framework derived directly from the underlying physics of the continuous propagator.

In this paper, we propose the Local Matrix Exponential Propagator (LMEP), a unified framework that reconstructs global time-steppers from locally exact evolution operators. Our approach is based on a profound yet computationally simple observation: the matrix exponential of a discrete operator, when restricted to a local polynomial subspace, automatically generates the high-order Taylor corrections required for both spatial and temporal accuracy. This framework effectively unifies the disparate perspectives of finite differences and semi-Lagrangian transport. We establish two fundamental theoretical pillars for this framework:
\begin{enumerate}
\item \textbf{Theorem \ref{th:isomorphism}:} For pure advection, we prove that the LMEP is algebraically isomorphic to optimal Lagrange-based semi-Lagrangian transport \cite{Staniforth1991, Steinstraesser2025}. This allows for high-CFL integration using purely sparse, banded algebraic operations.
\item \textbf{Theorem \ref{th:commutator}:} For mixed linear operators (e.g., advection-diffusion), we demonstrate that the LMEP provides the exact coupled evolution for the local interpolant. This bypasses the theoretical accuracy limits of traditional operator-splitting by resolving the "coupling" within the local exponential evaluation \cite{Siqi2025}.
\end{enumerate}

Finally, we extend the LMEP to semi-linear systems through an Augmented Local Matrix approach. This technique allows for the simultaneous "harvesting" of the entire hierarchy of ETD $\phi$-functions, ensuring that nonlinear forcing terms are integrated with the same stability characteristics as the linear propagator \cite{Kassam2005}. By restricting all spectral evaluations to a stencil of size $n \ll N$, we maintain a strictly banded structure and $O(N)$ efficiency. We validate the framework across a series of demanding benchmarks, including the viscous Burgers', Korteweg-de Vries (KdV), and Allen-Cahn equations, demonstrating preserved high-order accuracy and optimal scaling on both periodic and non-uniform Chebyshev grids.

\section{The Unified Generator Framework}
\label{sec:framework}

Consider the PDE problem
\begin{equation}\label{eq:pde}
    \frac{\partial u}{\partial t} = \mathcal{L} u, \quad x \in \Omega, \; t > 0,
\end{equation} 
with compatible boundary conditions, where $\mathcal{L}$ denotes the spatial differential operator. This operator may represent linear advection, diffusion, dispersion, or any linear combination thereof.

To construct a numerical approximation for equation \eqref{eq:pde}, we first discretize the spatial domain $\Omega$ into a grid $I$ consisting of $N$ points, replacing the continuous operator $\mathcal{L}$ with an $N\times N$ differentiation matrix $\mathbf{D}$. This reduces the PDE to a \textit{semi-discrete system} of ordinary differential equations:
\begin{equation}
    \frac{d\mathbf{u}}{dt} = \mathbf{D} \mathbf{u},
\end{equation}
where $u_j(t) \approx u(x_j, t)$ for grid points $x_j \in I$. The exact analytical solution of this semi-discrete system is formally given by the global linear operator 
\begin{equation}
\mathbf{u}(t+\Delta t) = \exp(\Delta t \mathbf{D})\mathbf{u}(t).
\end{equation} 
While mathematically exact, this global matrix exponential is computationally intractable for large-scale systems due to the ``fill-in'' phenomenon \cite{Moler2003}. The exponential of a sparse differentiation matrix $\mathbf{D}$ becomes a generally dense $N \times N$ matrix, requiring $\mathcal{O}(N^3)$ operations to compute and $\mathcal{O}(N^2)$ memory to store.

\subsection{Local Matrix Exponential Propagator}
To overcome this bottleneck while retaining exponential accuracy, we propose the \textit{Local Matrix Exponential Propagator (LMEP)}. Our methodology reconstructs the global time-stepper by assembling locally exact evolution operators. The construction proceeds as follows:

\begin{enumerate}
    \item \textbf{Localization:} We restrict the evolution to a local stencil $S_n \subset I$ of $n$ nodes. Let $\mathbf{D}_n \in \mathbb{R}^{n \times n}$ be the standard finite difference matrix approximating the spatial operator, with weights computed via Fornberg's recursive algorithms to ensure high-order accuracy on arbitrary spaced grids \cite{Fornberg1988}.
    
    \item \textbf{Local Evolution:} We compute the local propagator $\mathbf{W}_n(\Delta t) = \exp(\Delta t \mathbf{D}_n)$. This $n \times n$ operator captures the exact analytical evolution of the stencil data projected onto the subspace of polynomials $\mathcal{P}_{n-1}$.
    
    \item \textbf{Harvesting and Assembly:} To update a specific grid point $x_i$, we extract the corresponding row vector from $\mathbf{W}_n$ (typically the center row for interior points or edge rows for boundaries). These weights are then assembled into a sparse global integrator matrix with bandwidth $n$.
\end{enumerate}

This approach effectively allows us to ``harvest'' the spectral accuracy of the matrix exponential while enforcing a sparse sparsity pattern. We now demonstrate that this algebraic exponentiation naturally recovers and unifies optimal geometric schemes.

\subsection{Advection and Algebraic Isomorphism}
Consider the scalar linear advection equation $u_t + a u_x = 0$. The corresponding generator is $\mathbf{W}_n = \exp(-a \Delta t \mathbf{D}_n)$.


\begin{theorem}[Isomorphism to Lagrange Interpolation]\label{th:isomorphism}
Let $\mathbf{u} \in \mathbb{R}^n$ represent the values of a polynomial $p(x) \in \mathcal{P}_{n-1}$ on the stencil. The vector $\mathbf{v} = \exp(-a \Delta t \mathbf{D}_n)\mathbf{u}$ contains the exact values of the shifted polynomial $p(x - a\Delta t)$. Consequently, the weights of the LMEP are identical to Lagrange interpolation weights.
\end{theorem}

\begin{proof}
Let $\Phi: \mathcal{P}_{n-1} \to \mathbb{R}^n$ be the evaluation isomorphism $\Phi(p) = [p(x_1), \dots, p(x_n)]^T$.
Since $\mathbf{D}_n$ is exact on $\mathcal{P}_{n-1}$ \cite{Fornberg1998}, it represents the derivative operator $\partial_x$ in this basis: $\mathbf{D}_n^{(k)} \mathbf{u} = \Phi(p^{(k)}(x))$.
Substituting this into the series expansion of the matrix exponential:
\begin{equation}
    \exp(-a \Delta t \mathbf{D}_n)\mathbf{u} = \sum_{k=0}^{\infty} \frac{(-a \Delta t)^k}{k!} \Phi\left( p^{(k)}(x) \right) = \Phi\left( \sum_{k=0}^{n-1} \frac{(-a \Delta t)^k}{k!} p^{(k)}(x) \right).
\end{equation}
The infinite sum terminates at $k=n-1$ due to the nilpotency of differentiation on polynomials of degree $n-1$. The remaining sum is exactly the finite Taylor expansion $p(x - a\Delta t)$. Thus, the LMEP operator performs exact semi-Lagrangian transport on the polynomial basis.
\end{proof}


\begin{example}[Automatic Generation of Advective Schemes]\label{ex:advection}
Consider a standard 3-node stencil $x \in \{-h, 0, h\}$ and the centered finite difference matrix approximating the first derivative:
$$\mathbf{D}_3^{(1)} = \frac{1}{h}
\begin{pmatrix} 
    -1.5 & 2 & -0.5 \\
     -0.5 & 0 & 0.5 \\
      0.5 & -2 & 1.5 
\end{pmatrix}. $$
Traditionally, deriving time-accurate schemes involves a tedious process: expanding $u(t+\Delta t)$ in a Taylor series, substituting spatial derivatives ($u_{tt} = a^2 u_{xx}$), and manually discretizing to cancel error terms.

In the LMEP framework, we simply evaluate $\mathbf{W}_3 = \exp(\nu h \mathbf{D}_3^{(1)})$ where $\nu = -a\Delta t / h$. This single algebraic operation generates valid schemes for every node:

\begin{itemize}
\item\textbf{Interior Scheme (Lax-Wendroff):} 
Extracting the center row of $\mathbf{W}_3$ automatically yields:
$$ \mathbf{w}_{\text{center}} = \left[ \frac{\nu(\nu-1)}{2}, \quad 1 - \nu^2, \quad \frac{\nu(\nu+1)}{2} \right]. $$
This matches the classic Lax-Wendroff coefficients exactly \cite{LeVeque2007}.

\item \textbf{Boundary Scheme (Beam-Warming):} 
Extracting the last row (corresponding to the downwind boundary) automatically yields:
$$ \mathbf{w}_{\text{right}} = \left[ \frac{\nu(\nu+1)}{2}, \quad -\nu(\nu+2), \quad \frac{(\nu+1)(\nu+2)}{2} \right]. $$
This recovers the second-order upwind (Beam-Warming) scheme \cite{LeVeque2007}, which is crucial for stable boundary updates. By symmetry, extracting the first row yields the corresponding one-sided scheme for the left boundary.
\end{itemize} 
\end{example}


\begin{remark}[Exactness via Nilpotency] In Example \ref{ex:advection} the matrix $\mathbf{W}_3$ is exact because the series $\sum (\nu \mathbf{D}_3^{(1)})^k/k!$ terminates naturally: the nilpotency of $\mathbf{D}_3^{(1)}$ on polynomials of degree $2$ ensures the LMEP automatically generates the correct high-order stabilizing terms.
\end{remark}

\subsection{Generalization to Diffusive and Mixed Operators}

We now generalize this result, establishing that the matrix exponential of the discrete operator yields the algebraically exact evolution for the local polynomial interpolant of any linear differential equation with constant coefficients.


\begin{theorem}[General Exactness]\label{th:commutator}
Let $\mathcal{L}$ be any linear differential operator with constant coefficients. Let $\mathbf{u} \in \mathbb{R}^n$ be the nodal values at time $t$, and let $p(x) \in \mathcal{P}_{n-1}$ be the unique polynomial interpolant of this data. The discrete evolution
\begin{equation}
\mathbf{u}(t + \Delta t) = \exp(\Delta t \mathbf{L}_n) \mathbf{u}
\end{equation}
yields the exact nodal values of the analytical solution to the Cauchy problem $\partial_t u = \mathcal{L} u$ with initial condition $u(x,t) = p(x)$.
\end{theorem}

\begin{proof}
Let $\Phi: \mathcal{P}_{n-1} \to \mathbb{R}^n$ be the evaluation isomorphism. We decompose the operator into a scalar reaction term and a purely differential part: $\mathcal{L} = c_0 + \mathcal{L}_{\text{diff}}$. Correspondingly, the discrete matrix is $\mathbf{L}_n = c_0 \mathbf{I} + \mathbf{K}_n$. 

Since the identity matrix commutes with $\mathbf{K}_n$, the exponential factorizes as $\exp(\Delta t \mathbf{L}_n) = e^{c_0 \Delta t} \exp(\Delta t \mathbf{K}_n)$. The matrix $\mathbf{K}_n$ is nilpotent on the subspace $\Phi(\mathcal{P}_{n-1})$ because it represents strictly differential operators that eventually annihilate polynomials of degree $n-1$ \cite{Fornberg1988}. Consequently, the series expansion terminates naturally:
\begin{equation}
\exp(\Delta t \mathbf{K}_n)\mathbf{u} = \sum_{k=0}^{n-1} \frac{\Delta t^k}{k!} \mathbf{K}_n^k \Phi(p) = \Phi\left( \sum_{k=0}^{n-1} \frac{\Delta t^k}{k!} \mathcal{L}_{\text{diff}}^k p(x) \right) = \Phi\left( e^{\Delta t \mathcal{L}_{\text{diff}}} p(x) \right).
\end{equation}
Multiplying by the scalar factor $e^{c_0 \Delta t}$ recovers the full analytical solution $\Phi(e^{\Delta t \mathcal{L}} p(x))$, completing the proof.
\end{proof}



\begin{remark}[No Operator Splitting]
A crucial advantage of the LMEP is the handling of mixed operators, such as $u_t + a u_x = \nu u_{xx}$. Standard methods often split this into separate advection and diffusion steps \cite{Strang1968}, introducing significant numerical artifacts. As noted by Hundsdorfer and Verwer \cite{Hundsdorfer2003}, these splitting errors often dominate the global error in multi-physics simulations. The unified generator $\mathbf{W}_n = \exp(\Delta t (-a\mathbf{D}_n^{(1)} + \nu\mathbf{D}_n^{(2)}))$ captures the full algebraic coupling of transport and spreading within the local stencil, bypassing the $\mathcal{O}(\Delta t^p)$ commutator errors entirely.
\end{remark}


\begin{example}[Automatic Generation of Diffusive Schemes]\label{ex:diffusion}
Consider the diffusion equation $u_t = \nu u_{xx}$. Traditionally, constructing high-order schemes requires a tedious derivation: observing that $u_t = \nu u_{xx}$ implies $u_{tt} = \nu^2 u_{xxxx}$, and manually discretizing the higher-order term to cancel temporal truncation errors \cite{LeVeque2007}.

In the LMEP framework, this complexity is replaced by evaluating $\mathbf{W}_n = \exp(\mu h^2 \mathbf{D}_n^{(2)})$ where $\mu = \nu \Delta t / h^2$. This single operation automatically generates the correct weights:

\begin{itemize}
\item \textbf{Standard Scheme (FTCS):} On a 3-node stencil $\{-h, 0, h\}$, the exponential of the second derivative matrix $\mathbf{D}_3^{(2)}$ recovers the standard Forward Time Centered Space (FTCS) scheme \cite{LeVeque2007} exactly:
$$ \mathbf{w}_{\text{center}} = [\mu, \quad 1 - 2\mu, \quad \mu]. $$
Because the 3-node stencil only resolves polynomials up to degree 2, the discrete fourth-derivative matrix is strictly zero ($(\mathbf{D}_3^{(2)})^2=0$), causing the exponential series to truncate exactly after the linear term.

\item \textbf{High-Order Stabilization:} On a 5-node stencil $\{-2h, \dots, 2h\}$, the LMEP automatically incorporates the necessary higher-order corrections. The center row of $\mathbf{W}_5 = \exp(\mu \mathbf{D}_5^{(2)})$ yields:
$$ \mathbf{w}_{\text{center}} = \left[ -\frac{\mu}{12} + \frac{\mu^2}{2}, \quad \frac{4\mu}{3} - 2\mu^2, \quad 1 - \frac{5\mu}{2} + 3\mu^2, \quad \frac{4\mu}{3} - 2\mu^2, \quad -\frac{\mu}{12} + \frac{\mu^2}{2} \right]. $$
Notice the automatic emergence of the $\mu^2$ terms. Traditionally, high-order diffusion schemes require matching spatial terms to the temporal error ($u_{tt} = \nu^2 u_{xxxx}$) via tedious manual derivation \cite{LeVeque2007}. Here, the matrix exponential series naturally generates the term $\frac{1}{2}(\mu \mathbf{D}_5^{(2)})^2$. This algebraic square provides the exact discrete approximation of the fourth-order derivative $\partial_{xxxx}$ required for stability \cite{Fornberg1998}, ensuring the spatial discretization automatically respects the physics of the governing PDE.
\end{itemize}
\end{example}

\section{Generalization to Semi-Linear PDEs}
\label{sec:etd}

We extend the LMEP framework to semi-linear equations of the form:
\begin{equation}
\frac{\partial u}{\partial t} = \mathcal{L}u + \mathcal{N}(u, t),
\end{equation}
where $\mathcal{L}$ is a linear operator and $\mathcal{N}$ is a nonlinear forcing term.
Standard Exponential Time Differencing (ETD) methods integrate the linear part exactly using Duhamel's formula:
\begin{equation}
u(t_{k+1}) = e^{\mathcal{L}\Delta t}u(t_k) + \int_{0}^{\Delta t} e^{\mathcal{L}(\Delta t - \tau)} \mathcal{N}(u(t_k + \tau), t_k + \tau) \, d\tau.
\end{equation}
The core challenge is evaluating the integral involving the matrix exponential and the nonlinear term. 
Since the future evolution of $\mathcal{N}$ is unknown, the strategy of ETD is to approximate the nonlinearity as a polynomial in time and integrate it against the matrix exponential. This process naturally generates a hierarchy of operators known as $\phi$-functions, defined recursively to resolve the singularity at $z=0$:
\begin{equation}
\phi_0(z) = e^z, \quad \phi_{j+1}(z) = \frac{\phi_j(z) - 1/j!}{z}, \quad j \ge 0.
\end{equation}
These functions represent the exact response of the linear system to polynomial forcing. Based on how $\mathcal{N}$ is approximated, two primary classes of schemes arise:

\begin{itemize}
\item \textbf{ETD Multistep Methods (ETD-$s$):} These extend Adams-Bashforth schemes by using backward differences ($\nabla$) to interpolate the history of $\mathcal{N}$ \cite{Hundsdorfer2003}. The general order-$s$ update is:
\begin{equation}
u_{k+1} = e^{\mathcal{L}\Delta t}u_k + \Delta t \sum_{j=0}^{s-1} \phi_{j+1}(\mathcal{L}\Delta t) \nabla^j \mathcal{N}(u_k,t_k).
\end{equation}

\item
\textbf{ETD Runge-Kutta (ETDRK4):} To improve stability and avoid startup costs, Runge-Kutta variants evaluate $\mathcal{N}$ at intermediate substeps \cite{Kassam2005}. The classic fourth-order scheme computes stages $\mathbf{a}, \mathbf{b}, \mathbf{c}$:
\begin{align*}
a_k &= e^{\mathcal{L}\Delta t/2} u_k + \frac{\Delta t}{2} \phi_1(\tfrac{\mathcal{L}\Delta t}{2}) \mathcal{N}(u_k,t_k), \\
b_k &= e^{\mathcal{L}\Delta t/2} u_k + \frac{\Delta t}{2} \phi_1(\tfrac{\mathcal{L}\Delta t}{2}) \mathcal{N}(a_k, t_k + \tfrac{\Delta t}{2}), \\
c_k &= e^{\mathcal{L}\Delta t/2} a_k + \frac{\Delta t}{2} \phi_1(\tfrac{\mathcal{L}\Delta t}{2}) \left( 2\mathcal{N}(b_k, t_k + \tfrac{\Delta t}{2}) - \mathcal{N}(u_k, t_k) \right).
\end{align*}
The final update combines these using weights constructed from higher-order $\phi$-functions (where $\phi_j = \phi_j(\mathcal{L}\Delta t)$):
\begin{equation}
u_{k+1} = e^{\mathcal{L}\Delta t} u_k + \Delta t \left[ \alpha \mathcal{N}(u_k) + \beta (\mathcal{N}(a_k) + \mathcal{N}(b_k)) + \gamma \mathcal{N}(c_k) \right],
\end{equation}
with coefficients defined to cancel stiff error terms up to fourth order:
$$\alpha = \phi_1 - 3\phi_2 + 4\phi_3, \quad \beta = 2\phi_2 - 4\phi_3, \quad \gamma = 4\phi_3 - \phi_2.$$

\end{itemize}

A major challenge in ETD implementations is the numerical instability of recursive $\phi$-function formulas near $z=0$ due to catastrophic cancellation errors \cite{Kassam2005}. To circumvent this, one common approach is the \textit{Augmented Matrix} method \cite{Sidje1998, Moler2003}. This approach computes all necessary operators simultaneously by exponentiating a higher-dimensional block matrix composed of $N \times N$ identity and zero blocks \cite{AlMohy2011}. For a method requiring $\{\phi_0, \dots, \phi_3\}$, the augmented operator and its exponential are:
\begin{equation}\label{eq:augmented_matrix}
\tilde{\mathcal{A}} = 
\begin{pmatrix}
\mathcal{L} & I & 0 & 0 \\
0 & 0 & I & 0 \\
0 & 0 & 0 & I \\
0 & 0 & 0 & 0
\end{pmatrix} 
\implies
\exp( \tilde{\mathcal{A}} \Delta t ) = 
\begin{pmatrix}
e^{\mathcal{L}\Delta t} & \Delta t \phi_1 & \Delta t^2 \phi_2 & \Delta t^3 \phi_3 \\
0 & I & \Delta t I & \frac{\Delta t^2}{2} I \\
\vdots & \vdots & \vdots & \vdots 
\end{pmatrix}.
\end{equation}
The first block-row naturally contains the linear propagator and the exact $\phi$-functions scaled by powers of $\Delta t$. This allows for the "harvesting" of all integration weights in a single operation, avoiding explicit recurrence relations entirely.

While analytically elegant, applying this technique globally presents a severe computational barrier. For a spatial domain with $N$ degrees of freedom, the augmented operator $\tilde{\mathcal{A}}$ has dimensions $sN \times sN$ (where $s$ is the number of required $\phi$-functions). Computing the exponential of this massive matrix scales as $\mathcal{O}((sN)^3)$ and results in a fully dense operator \cite{Moler2003, Sidje1998}. This destroys the sparsity of the spatial discretization and makes the standard augmented matrix approach prohibitively expensive for large-scale systems.

\subsection{The Augmented LMEP Framework}
\label{sec:augmented_lmep}
To harness the stability of the augmented matrix method without the crippling global cost, we introduce the \textit{Augmented Local Matrix}. By restricting the spectral evaluation to the polynomial subspace of the stencil, we reduce the problem dimension from $sN$ to $sn$ (where $n \ll N$).

\begin{enumerate}
\item \textbf{Augmented Stencil Definition:} For a local stencil of size $n$ with discrete linear operator $\mathbf{L}_n$, we construct a block matrix $\tilde{\mathbf{A}}_n$ of dimension $sn \times sn$. For a scheme requiring up to $\phi_{s-1}$, the generator is:
\begin{equation}
\tilde{\mathbf{A}}_n = \begin{pmatrix}
\mathbf{L}_n & \mathbf{I}_n & \mathbf{0} & \dots \\
\mathbf{0} & \mathbf{0} & \mathbf{I}_n & \dots \\
\vdots & \vdots & \vdots & \ddots \\
\mathbf{0} & \mathbf{0} & \mathbf{0} & \mathbf{0}
\end{pmatrix}.
\end{equation}

\item \textbf{Simultaneous Weight Harvesting:} Exponentiating this block matrix, $\mathbf{W}_n^{\text{aug}} = \exp(\Delta t \tilde{\mathbf{A}}_n)$, automatically generates the exact linear propagator and the complete hierarchy of $\phi$-functions (where $\phi_j = \phi_j(\Delta t \mathbf{L}_n)$) in the first block-row:
\begin{equation}
\mathbf{W}_n^{\text{aug}} =
\begin{pmatrix}
e^{\Delta t \mathbf{L}_n} & \Delta t \phi_1 & \Delta t^2 \phi_2 & \dots \\ 
\mathbf{0} & \mathbf{I}_n & \Delta t \mathbf{I}_n & \dots \\ 
\vdots & \vdots & \vdots & \ddots 
\end{pmatrix}. 
\end{equation} 
By extracting the rows corresponding to the stencil center from these blocks, we obtain a set of weight vectors $\{\mathbf{w}_{\text{lin}}, \mathbf{w}_{\phi_1}, \mathbf{w}_{\phi_2}, \dots\}$. These weights allow us to construct high-order semi-linear integrators that are as sparse as the underlying stencil.

\item  \textbf{Localized Implementation:} Once the weight vectors $\{\mathbf{w}_{\text{lin}}, \mathbf{w}_{\phi_1}, \dots\}$ are harvested, the implementation of specific time-stepping schemes is immediate. We simply substitute the global operators in the standard definitions (e.g., the multistep update or the ETDRK4 stages $\mathbf{a}, \mathbf{b}, \mathbf{c}$) with their corresponding local discrete counterparts. Whether integrating the history of the nonlinearity (ETD-$s$) or computing intermediate Runge-Kutta stages, every term is evaluated via dot products on the local stencil. This yields a node-by-node update that inherits the high-order temporal stability of the underlying spectral scheme while strictly enforcing a spatial dependency of only $n$ neighbors.

\end{enumerate}

\section{Numerical Validations}
\label{sec:experiments}

We validate the theoretical foundations of the LMEP framework through a series of reproducible benchmarks, progressing from linear transport to stiff, nonlinear dynamics. These experiments are specifically structured to verify the core pillars of our approach: (i) the algebraic isomorphism to Semi-Lagrangian methods for high-CFL advection (Theorem \ref{th:isomorphism}), (ii) the elimination of operator-splitting errors in coupled linear systems (Theorem \ref{th:commutator}), and (iii) the preservation of high-order temporal accuracy within the localized ETD formulation for nonlinear PDEs.

\subsection{High-CFL Transport and Spectral Semi-Lagrangian Isomorphism}
\label{sec:advection}

To empirically validate the algebraic isomorphism established in Theorem \ref{th:isomorphism}, we simulate the 1D linear advection equation with periodic boundary conditions,
\begin{equation}
u_t + a u_x = 0,  \quad x \in [-1, 1], 
\end{equation} 
with a characteristic wave speed $a=1$. The initial condition is a sharp Gaussian profile $u(x,0) = \exp(-40x^2)$, tracked over 100 full periods to evaluate the long-term accumulation of numerical errors.

We evaluate the Local Matrix Exponential Propagator (LMEP) across two distinct regimes: a Lagrangian configuration ($\nu \in \mathbb{Z}$) and a non-Lagrangian configuration ($\nu \in \mathbb{Z} + 0.5$). For centered stencils, the Courant number is $\nu = 1 - \sigma$, while for one-sided (fully upwind) stencils, we deliberately test extreme CFL values of $\nu = \frac{n+1}{2} - \sigma$, where $\sigma$ represents the sub-grid shift.

As illustrated in Fig.~\ref{fig:combined_conv}, the Lagrangian configurations for both stencil types exhibit an immediate error floor near machine precision ($\approx 10^{-12}$), independent of the stencil size $n$. This confirms Theorem \ref{th:isomorphism}: when the physical characteristic lands exactly on a grid node, the matrix exponential recovers the discrete identity shift perfectly, effectively performing exact semi-Lagrangian transport \cite{Staniforth1991}. In the non-Lagrangian case ($\sigma=0.5$), the characteristic falls midway between nodes, requiring interpolation. Here, the LMEP demonstrates strict spectral convergence as $n$ increases, acting as an optimal localized spectral filter \cite{Fornberg1998,Trefethen1996}.

A comparison of the spectral footprints in Fig.~\ref{fig:spectral_centered} and \ref{fig:spectral_onesided} reveals a fundamental difference in stability characteristics. Centered stencils (Fig.~\ref{fig:spectral_centered}) exhibit near-neutral stability, with the amplification factor $|G|$ remaining remarkably close to unity across the majority of the resolved spectrum. This allows centered stencils to maintain high-fidelity, non-dissipative transport even at high orders.

In contrast, one-sided stencils (Fig.~\ref{fig:spectral_onesided}) naturally introduce numerical dissipation—manifesting as a characteristic "hump" in the dissipation plots—which serves to stabilize high-CFL upwind transport. Crucially, because the stable Courant number $\nu$ scales linearly with the stencil size ($n$), the permissible time step $\Delta t$ also increases proportionally. This stability is particularly impressive: a CFL of $\nu \approx 12.5$ for $n=25$ far exceeds the $\nu \le 1$ limit of standard explicit upwind schemes \cite{LeVeque2007}. However, this stability is finite; as the stencil size reaches $n=25$, the one-sided configuration begins to exhibit high-frequency oscillations. This behavior indicates the onset of the Runge phenomenon at the stencil edges, where the high-degree polynomial interpolant becomes hyper-sensitive. Based on these observations, we find that stencil sizes in the range $n \in [7, 15]$ provide the optimal balance between spectral-like accuracy and robust $\ell_\infty$ stability.

A major advantage of the LMEP framework is that this spectral-quality evolution is achieved through high-performance algebraic sparsity. While traditional global spectral methods require dense $N \times N$ matrices ($\mathcal{O}(N^2)$), the LMEP localizes the matrix exponential calculation to a small $n \times n$ stencil. When assembled, the resulting global evolution matrix is strictly banded with a bandwidth of $n$.

For stationary grids, these unique weights are harvested once during an initialization phase and stored. Consequently, while the per-step cost of the propagator scales as $\mathcal{O}(nN)$, the total execution time for one-sided stencils does not scale with $n$. Because the time step $\Delta t \propto n$, the number of steps required to reach a fixed final time scales as $\mathcal{O}(1/n)$, effectively canceling out the increased per-step cost (as demonstrated by the efficiency metrics in Fig.~\ref{fig:full_stencil_analysis}). This provides the accuracy of geometric semi-Lagrangian tracking with the efficiency of sparse finite-difference solvers, without the typical runtime penalties associated with high-order methods. This sparsity is maintained even for mixed operators, as the LMEP "harvests" the combined physics into a single set of sparse weights, bypassing the need for computationally expensive operator splitting.

\begin{figure}[htbp]
    \centering
    \begin{subfigure}[b]{1.0\textwidth}
        \centering
        \includegraphics[width=1\textwidth]{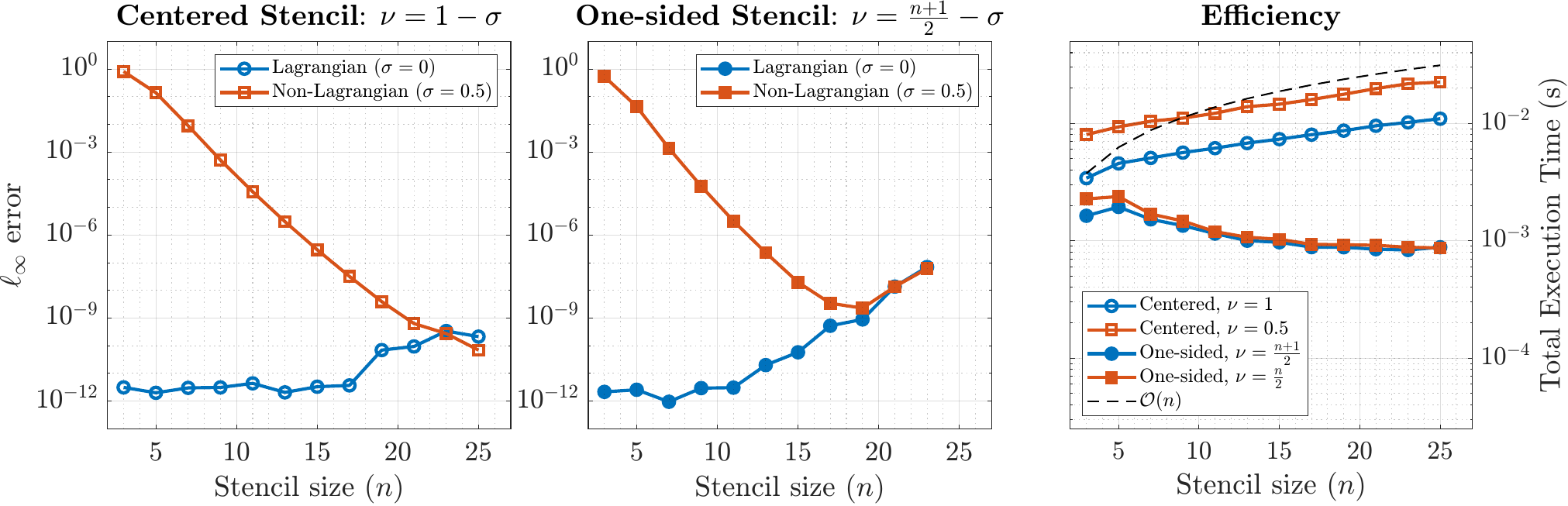}
        \caption{Global $\ell_\infty$ convergence and computational efficiency.} 
        \label{fig:combined_conv}
    \end{subfigure}
    
    \vspace{1.5em}
    
    \begin{subfigure}[b]{1.0\textwidth}
        \centering
        \includegraphics[width=0.95\textwidth]{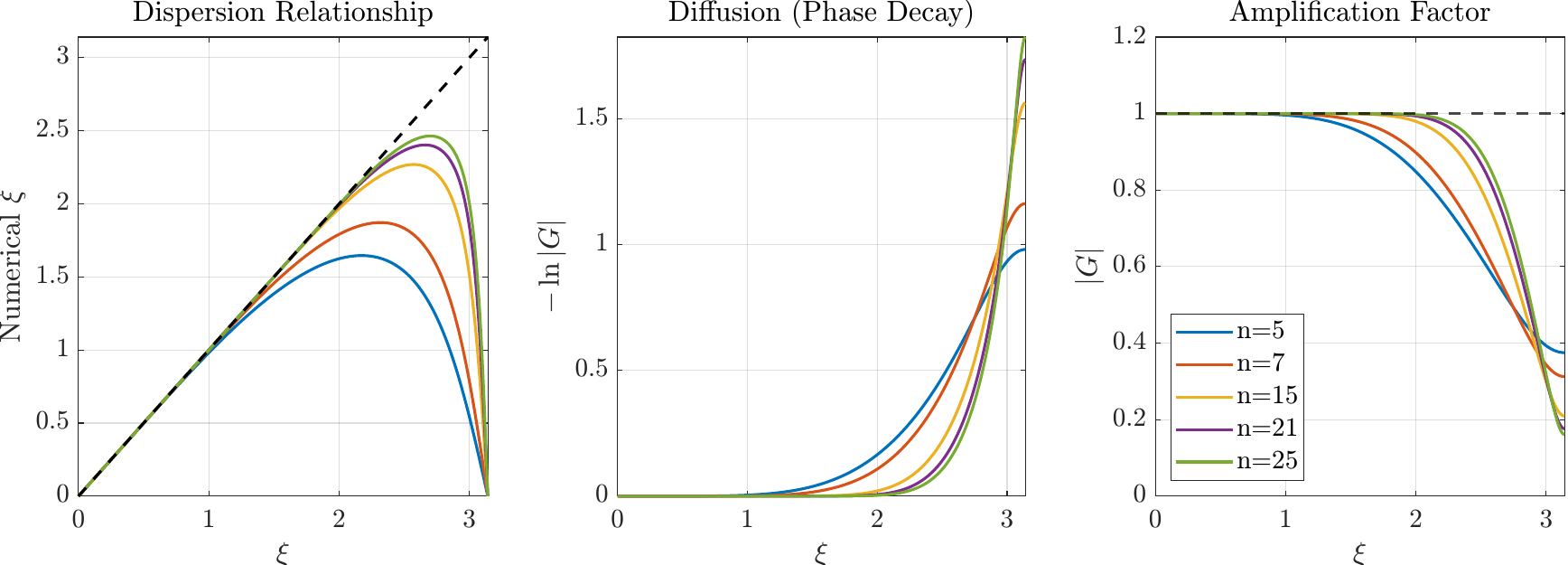}
        \caption{Spectral properties for Non-Lagrangian Centered stencils ($\sigma=0.5, \nu=0.5$): Near-neutral stability across all $n$.}
        \label{fig:spectral_centered}
    \end{subfigure}
    
    \vspace{1.5em}
    
    \begin{subfigure}[b]{1.0\textwidth}
        \centering
        \includegraphics[width=0.95\textwidth]{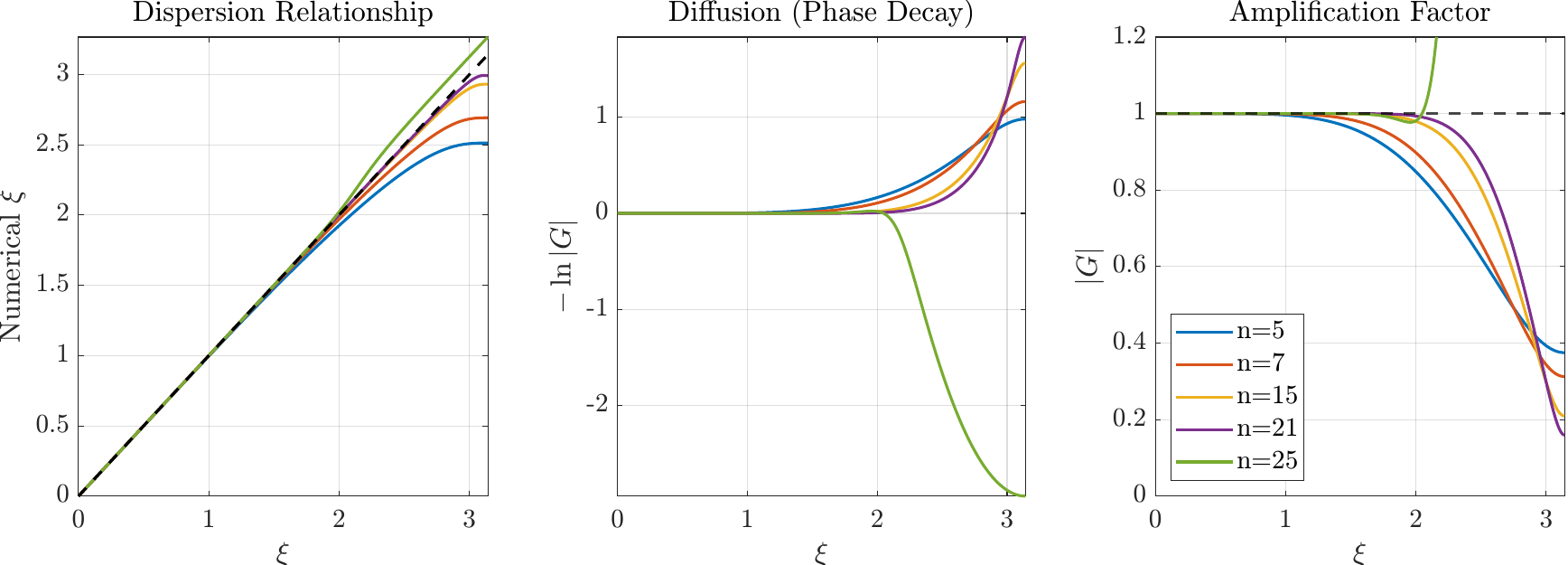}
        \caption{Spectral properties for Non-Lagrangian One-sided stencils ($\sigma=0.5, \nu = \frac{n}{2}$): Note the high-CFL stability limit and the onset of Runge-related noise at $n=25$.}
        \label{fig:spectral_onesided}
    \end{subfigure}
    
    \caption{Numerical Analysis of the Localized ETD Operator. Top: convergence reaching machine precision for Lagrangian transport. Middle/Bottom: spectral breakdown (dispersion and diffusion) for non-Lagrangian configurations. The one-sided configuration permits extremely high CFL conditions ($\nu \gg 1$) but encounters stability limits at $n=25$ due to edge-point sensitivity. The included efficiency data highlights that total execution time remains flat as $n$ increases, owing to the $\Delta t \propto n$ scaling.
}
    \label{fig:full_stencil_analysis}
\end{figure}

\subsection{Exact Coupled Evolution: Bypassing Operator Splitting}\label{sec:coupled}

Theorem \ref{th:commutator} asserts that the LMEP framework captures the exact analytical evolution for any linear differential equation with constant coefficients, naturally avoiding the pitfalls of operator splitting. To validate this, we test the coupled periodic linear advection-diffusion equation 
\begin{equation}
u_t + a u_x = \nu u_{xx}, \quad x \in [0, 2\pi],
\end{equation}
with parameters $a = 1$ and $\nu = 0.1$. The initial condition is given by the Gaussian pulse $u(x,0) = \exp(-10(x-\pi)^2)$ and integrated in time until $t = 1$.

Traditional Method of Lines (MoL) approaches frequently isolate the advective and diffusive operators into fractional steps, introducing a persistent $\mathcal{O}(\Delta t^p)$ commutator splitting error \cite{Strang1968, Hundsdorfer2003} because the operators $[L_{adv}, L_{diff}]$ do not commute. By directly exponentiating the coupled discrete operator on a local stencil, the LMEP framework "harvests" the combined physics into a single set of sparse weights, entirely bypassing this degradation. Furthermore, while computing a global matrix exponential would traditionally result in a dense $N \times N$ matrix, the LMEP approach restricts the spectral evaluation to the local polynomial subspace. This results in a global evolution matrix that is strictly banded, preserving the $\mathcal{O}(nN)$ memory footprint and computational efficiency of standard finite differences.

To rigorously quantify the LMEP's discrete spatial error without contamination from boundary artifacts, we compare the numerical results against a highly accurate, wrapped periodic analytical solution. This ground truth is obtained by summing the infinite-domain analytical solution over $K=2$ periodic shifts, ensuring accuracy down to machine precision. Fig.~\ref{fig:unified_convergence} presents the $\ell_\infty$ error convergence landscapes across three spatial grid resolutions: $N \in \{64, 128, 512\}$. The contour maps vividly demonstrate that within the stable operating envelope (denoted by the region below the black $\Delta t^*$ stability boundary), the fully coupled LMEP method achieves spectral convergence strictly as a function of the stencil size $n$ \cite{Fornberg1998}. The complete absence of an operator-splitting error plateau implies that the accuracy of the integration is limited solely by the polynomial projection of the spatial stencil. This perfectly corroborates the analytical exactness guaranteed by Theorem \ref{th:commutator}, confirming that the coupled physics are resolved seamlessly and highly efficiently.

\begin{figure}
    \centering
    \includegraphics[width=0.8\textwidth]{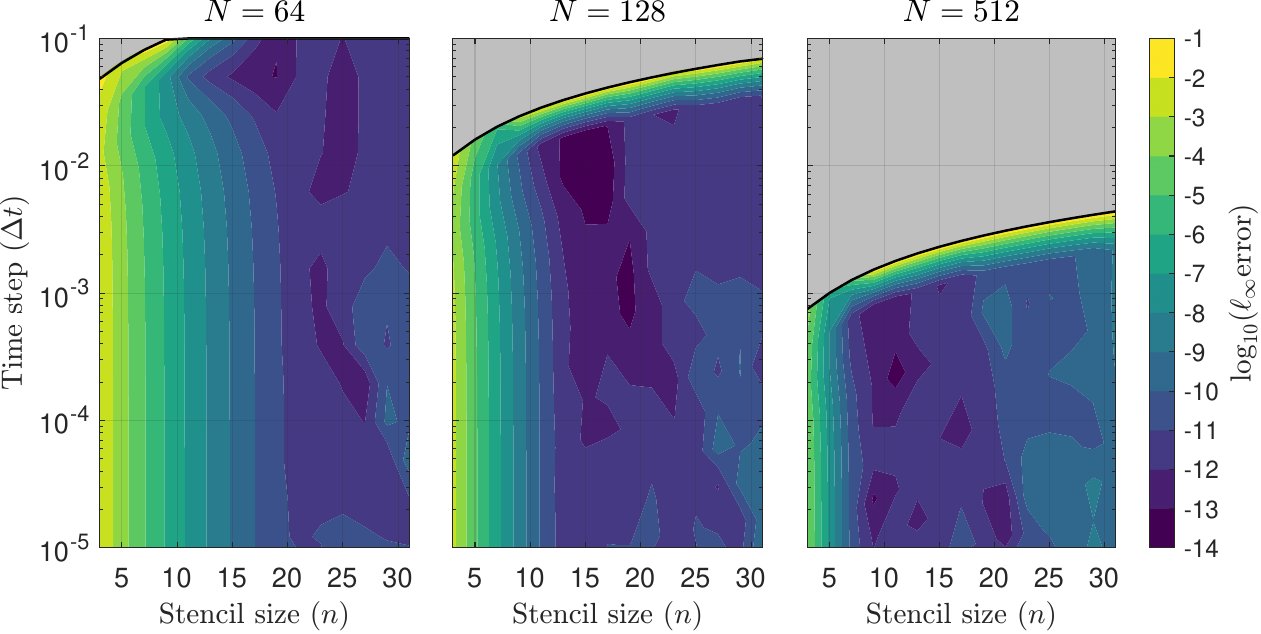}
    \caption{Error convergence landscapes for $N \in \{64, 128, 512\}$. The black line indicates the stability boundary $\Delta t^*$, showing the expansion of the stable operating envelope as a function of $n$. All plots share the same logarithmic scale to highlight the spectral convergence floor.}
    \label{fig:unified_convergence}
\end{figure}

\begin{remark}[Stability and Stencil Size]
The black boundary $\Delta t^*$ in Fig.~\ref{fig:unified_convergence} represents the empirical stability limit. In the diffusive regime, this manifests as a significant relaxation of the standard parabolic CFL condition $\Delta t \leq h^2 / (2\nu)$ \cite{LeVeque2007}. As seen in the maps, for $n=25$, the stable envelope extends nearly an order of magnitude beyond the traditional explicit limit, effectively bridging the gap toward unconditional stability without requiring implicit system solves.
\end{remark}

\subsection{Non-Linear PDEs}
The robustness of the Localized ETD formulation for stiff, non-linear systems is evaluated using three benchmarks from \cite{Kassam2005}. These tests verify that the high-order temporal accuracy of the underlying ETD scheme is preserved under the localized algebraic framework.

\subsubsection{Burgers' Equation}\label{sec:burgers}

We evaluate the robustness of the localized ETD framework using the viscous Burgers' equation on a periodic domain:
\begin{equation}
    u_t + u u_x = \nu u_{xx}, \quad x \in [-\pi, \pi),
\end{equation}
with a diffusion coefficient $\nu = 0.03$ and an initial Gaussian-like pulse $u(x,0) = \exp(-10\sin(x/2)^2)$. This equation serves as a canonical benchmark for non-linear advection-diffusion dynamics, assessing the solver's ability to resolve steep gradients without triggering the cell-Reynolds instabilities common in standard explicit schemes \cite{LeVeque2007}.

The system is integrated until $t=1$, a duration sufficient for non-linear steepening to challenge the scheme's stability. In our LMEP-ETD4 formulation, the stiff linear term $\nu u_{xx}$ is handled exactly within the local matrix exponential, while the non-linear term $N(u) = -u u_x$ is evaluated through the four stages of the ETD-RK4 scheme \cite{Kassam2005}. Crucially, the $\phi$-functions required for these stages are harvested using the localized augmented matrix approach (Section \ref{sec:augmented_lmep}), ensuring that the entire non-linear update is executed through a sequence of sparse, banded matrix-vector multiplications.

\begin{figure}
    \centering
    \includegraphics[width=0.8\textwidth]{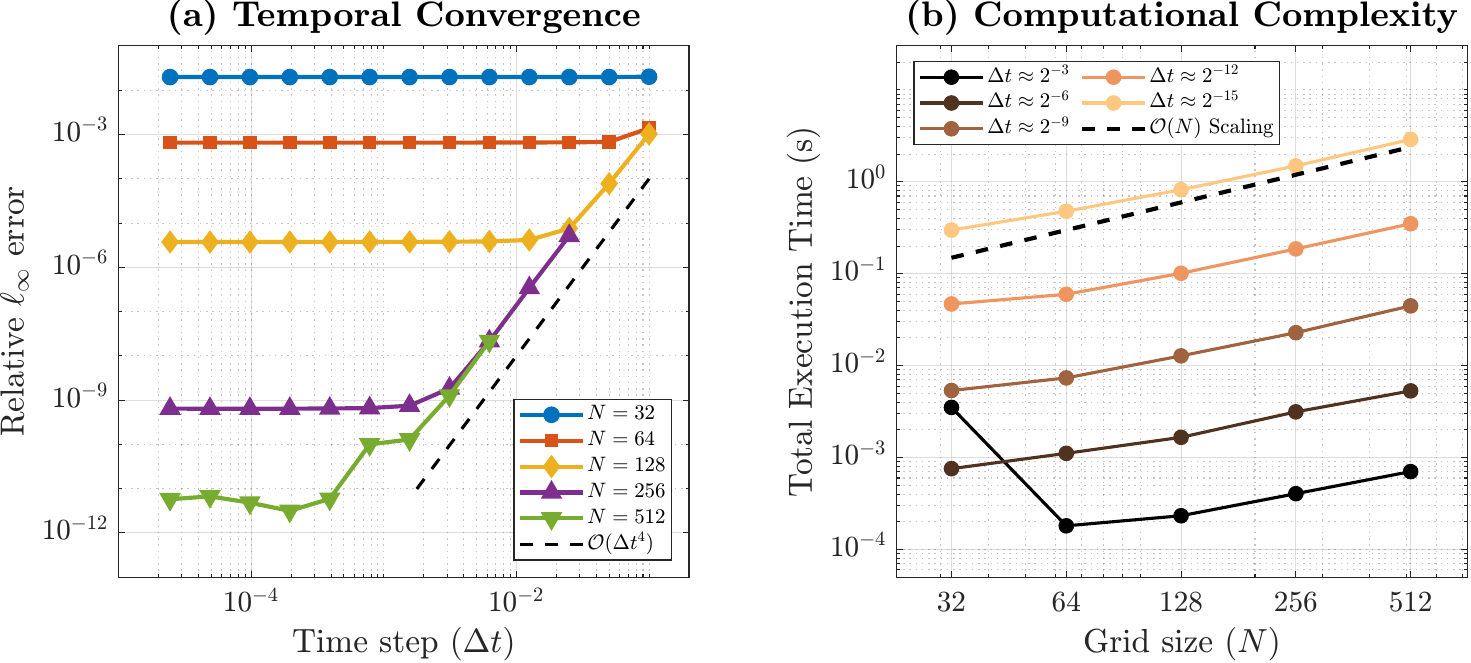}
    \caption{Numerical analysis for the viscous Burgers' equation with $\nu = 0.03$ and $n=19$. \textbf{(a)} Relative $\ell_\infty$ error convergence, demonstrating that the localized ETD-RK4 scheme preserves fourth-order accuracy in time ($\mathcal{O}(\Delta t^4)$) until the spatial error limit is reached. \textbf{(b)} Computational complexity, showing the total execution time scaling linearly ($\mathcal{O}(N)$) with grid resolution across five different time steps ($\Delta t$). This confirms the optimal efficiency of the sparse, banded operator implementation.}
    \label{fig:burgers_convergence}
\end{figure}

As shown in Fig.~\ref{fig:burgers_convergence}(a), the method exhibits the expected fourth-order temporal convergence ($\mathcal{O}(\Delta t^4)$) across a range of spatial resolutions $N$. The scheme follows the theoretical fourth-order slope until it reaches the spatial error floor determined by the resolution $N$ and stencil size $n$. This demonstrates that localizing the ETD operator does not degrade the high-order temporal features of the underlying integrator, even in the presence of non-linear steepening.

Furthermore, Fig.~\ref{fig:burgers_convergence}(b) empirically validates the computational efficiency of the localized framework. For a fixed stencil size ($n=19$), the total execution time scales linearly with the number of grid points, perfectly aligning with the theoretical $\mathcal{O}(N)$ reference line. This confirms that ``weight harvesting'' successfully translates the high-order accuracy of global ETD methods to a sparse context, where the non-linear updates are executed strictly through $\mathcal{O}(nN)$ banded matrix-vector operations. By avoiding the dense $\mathcal{O}(N^2)$ or $\mathcal{O}(N^3)$ operations typical of global matrix exponentials, the LMEP-ETD4 scheme achieves highly scalable performance without sacrificing spectral-like accuracy.

\subsubsection{Korteweg-de Vries (KdV) Equation}\label{sec:kdv}

To evaluate the performance of the localized ETD framework in the presence of higher-order dispersion, we solve the Korteweg-de Vries (KdV) equation:
\begin{equation}
u_{t} + u u_{x} + u_{xxx} = 0, \quad x \in [-\pi, \pi].
\end{equation}
Following the benchmark parameters from \cite{Kassam2005}, we initialize a two-soliton profile with $A=25$ and $B=16$. This problem serves as a rigorous test for the LMEP framework because the third-order dispersive operator creates a highly oscillatory unitary propagator, and the high-amplitude solitons challenge the spatial resolution. While global spectral methods are typically preferred here, our localized approach seeks to achieve comparable accuracy using weight harvesting from a local stencil of size $n=23$.

The simulation results ($t=0.006$) are presented in Fig.~\ref{fig:kdv_results}. The space-time evolution (top panel) illustrates the clean interaction between solitons. Crucially, the localized update captures the non-linear collision and phase shift without generating the parasitic dispersive ripples or "Gibbs-like" oscillations \cite{Trefethen2000} that typically plague standard finite difference approximations of high-order derivatives.

\begin{figure}
    \centering
    \includegraphics[width=0.95\textwidth]{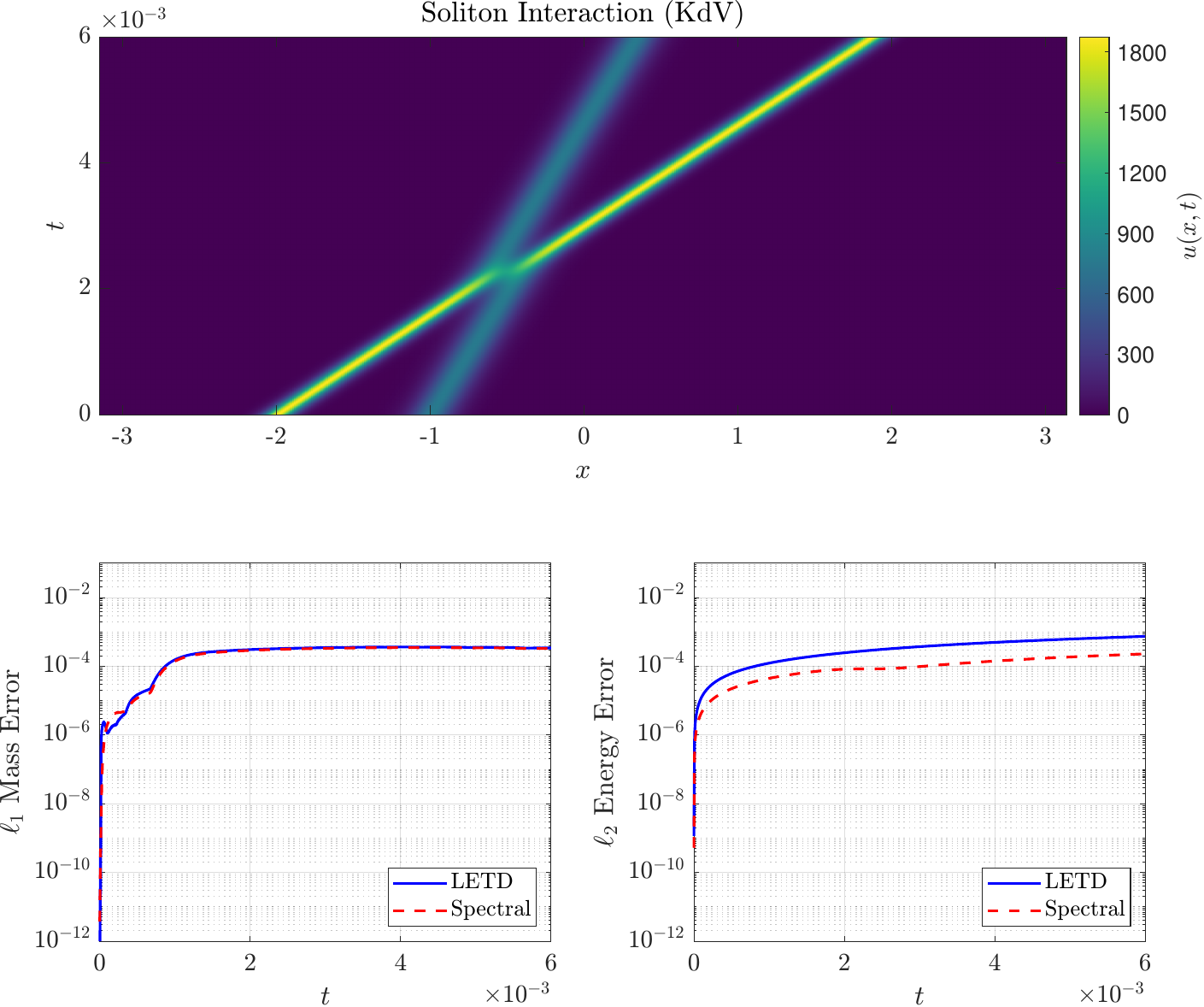}
    \caption{Numerical results for the KdV equation using parameters from \cite{Kassam2005}. Top: Space-time evolution $u(x,t)$ showing the two-soliton interaction and phase shift using LETD. Bottom-left: Comparison of mass conservation error ($\ell_1$ norm). Bottom-right: Comparison of energy conservation error ($\ell_2$ squared). The localized ETD method ($n=23$) demonstrates accuracy and conservation properties indistinguishable from the global spectral reference.}
    \label{fig:kdv_results}
\end{figure}

The conservation properties are analyzed in the bottom panels of Fig.~\ref{fig:kdv_results}, comparing the LETD method against a 512-point global Fourier spectral reference:
\begin{itemize}
\item \textbf{Mass Conservation ($L_1$ norm):} The $L_1$ error (bottom-left), defined as $I_1 = \int_{-\pi}^{\pi} |u| \, dx$, remains stable near $10^{-4}$. While standard mass conservation for the KdV equation monitors $\int_{-\pi}^{\pi} u \, dx$, we deliberately employ the absolute value to provide a more stringent test, ensuring that local dispersive errors do not artificially cancel across the domain.
\item \textbf{Energy Conservation ($L_2$ norm):} The energy error (bottom-right), defined as $I_2 = \int_{-\pi}^{\pi} u^2 \, dx$, also settles near $10^{-4}$.
\end{itemize}

The $10^{-4}$ error floor is consistent with the global spectral reference at this resolution ($N=512$), representing the spatial discretization limit of the grid rather than a temporal integration deficiency. The LETD method tracks the spectral reference exactly throughout the simulation with no evidence of secular error growth. This indicates that the localized operator effectively mimics the unitary nature of the analytical propagator. By restricting the $u_{xxx}$ evaluation to the local polynomial subspace, the LMEP framework achieves spectral-quality results while maintaining a strictly banded global matrix, bypassing the $\mathcal{O}(N^3)$ computational bottlenecks of global spectral methods.

\subsubsection{Allen-Cahn Equation}\label{sec:allen_cahn}

Finally, we examine the reaction-diffusion Allen--Cahn equation, a standard benchmark for evaluating stability in the presence of stiff nonlinearities and moving interfaces:
\begin{equation}
u_t = \epsilon u_{xx} + u - u^3, \quad x \in [-1, 1],
\end{equation}
where $\epsilon = 0.01$. We impose non-homogeneous Dirichlet boundary conditions $u(-1, t) = -1$ and $u(1, t) = 1$, initialized with the profile $u(x,0) = 0.53x + 0.47\sin(-1.5\pi x)$ \cite{Kassam2005}. The problem is solved on a Chebyshev grid ($N=64$), naturally suited for the domain's non-periodic nature and boundary-concentrated dynamics \cite{Trefethen2000}.

Because Chebyshev nodes are non-uniformly distributed, the local differentiation matrix $\mathbf{L}_n$ is unique for each grid point $x_i$. Crucially, this localized approach completely bypasses the need to compute the exponential of the global Chebyshev differentiation matrix, a notoriously ill-conditioned and dense $\mathcal{O}(N^3)$ operation. The LMEP framework adaptively harvests unique integration weights for every node, accounting for varying grid spacing without requiring coordinate transformations. Since the grid is stationary, these weights are harvested once and stored, reducing the per-step cost to sparse matrix-vector multiplication and maintaining $\mathcal{O}(N)$ efficiency.

The results ($n=21, t=70$) are shown in Fig.~\ref{fig:allen_cahn}. The evolution illustrates rapid phase separation into $u \approx \pm 1$ regions followed by a metastable period where interfaces move slowly toward annihilation (upper half of plot) \cite{Kassam2005}. Comparing against a global Chebyshev spectral reference ($N=1024, \Delta t = 10^{-4}$), the localized scheme achieves an $\ell_\infty$ error of $\approx 10^{-5}$ at $t=70$. This confirms that by treating the stiff diffusive term exactly via the local matrix exponential, the LETD framework effectively resolves the reaction-diffusion balance without the global overhead or restrictive stability limits of high-order finite differences.

\begin{figure}
\centering
\includegraphics[width=0.75\textwidth]{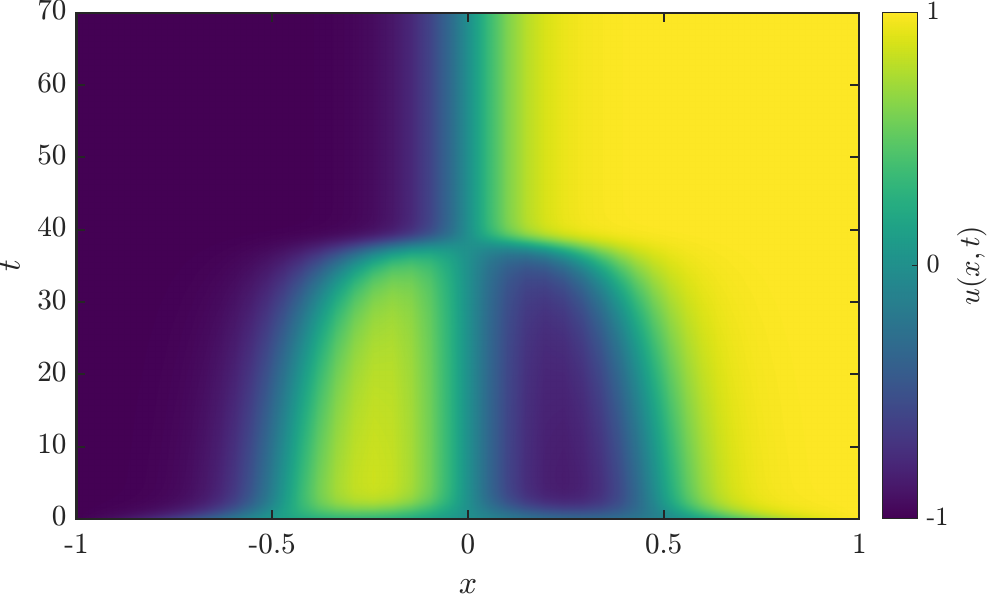}
\caption{Phase dynamics of the Allen--Cahn equation ($\epsilon = 0.01$) on a Chebyshev grid ($N=64$). The evolution captures the formation of stable equilibria and the phenomenon of metastability. Initial conditions and comparison benchmarks are adapted from \cite{Kassam2005}. The localized ETD method tracks the $N=1024$ global spectral reference with a final $\ell_\infty$ error of $10^{-5}$.}
\label{fig:allen_cahn}
\end{figure}

\section{Conclusion}
\label{sec:conclusion}
We have presented a unified algebraic framework for the construction of localized exponential integrators that effectively bridges the gap between spectral accuracy and finite-difference efficiency. By establishing the isomorphism between the discrete matrix exponential on a local polynomial subspace and optimal semi-Lagrangian transport (Theorem \ref{th:isomorphism}), we have provided a rigorous foundation for why localized weight harvesting recovers optimal geometric schemes.

We also demonstrated that this framework naturally bypasses the persistent commutator errors inherent in traditional operator-splitting methods by evaluating the coupled operator directly (Theorem \ref{th:commutator}). Our numerical validations—ranging from long-term advection over 100 periods to the stiff nonlinear dynamics of the Burgers', KdV, and Allen-Cahn equations—confirm that the Localized ETD formulation preserves high-order temporal accuracy while empirically achieving optimal $\mathcal{O}(N)$ computational complexity. Significantly, we identified a highly advantageous scaling property for upwind configurations: because the stable Courant number increases linearly with the stencil size, the total execution time remains flat even as spatial accuracy increases. 

The method exhibits robust performance on both periodic Fourier and non-periodic Chebyshev grids, confirming its versatility as a general-purpose solver for complex multi-physics systems. While this work establishes the spectral stability and high-CFL limits of the LMEP in one spatial dimension, future research will focus on the extension to higher-dimensional manifolds. In these regimes, the strict sparsity of the localized generator will offer the most profound computational advantages over traditional global spectral methods.

\section*{Software and Data Availability}
The LMEP framework, including weight harvesting routines and ETD-RK4 implementations, is open-source. The full suite of scripts required to reproduce the 1D linear advection, Burgers, KdV, and Allen--Cahn benchmarks—along with pre-computed localized weights—is available at \url{https://gitlab.com/vbrgit/LMEP-ETD.git}.

\bibliography{references}  

\end{document}